\documentclass[11pt]{article}

\usepackage[a4paper,margin=1in]{geometry}
\usepackage{amsmath,amssymb,amsthm}
\usepackage{graphicx}
\usepackage{hyperref}

\hypersetup{
 colorlinks=true,
 linkcolor=blue,
 citecolor=blue,
 urlcolor=blue,
 pdftitle={Higher-Order Hankel Obstructions to Free Infinite Divisibility for Beta Distributions},
 pdfauthor={Diwen Yu},
 pdfkeywords={free infinite divisibility, beta distribution, free cumulants, Hankel matrix}
}

\newtheorem{theorem}{Theorem}[section]
\newtheorem{proposition}[theorem]{Proposition}
\newtheorem{corollary}[theorem]{Corollary}
\newtheorem{lemma}[theorem]{Lemma}
\theoremstyle{definition}

\newcommand{\FID}{\mathrm{FID}}
\newcommand{\R}{\mathbb{R}}
\newcommand{\N}{\mathbb{N}}
\newcommand{\dd}{\,\mathrm{d}}

\title{Higher-Order Hankel Obstructions to Free Infinite Divisibility\
for Beta Distributions}
\author{Diwen Yu\\
\small Tsinghua University\\
\small \texttt{ydw23@mails.tsinghua.edu.cn
}}
\date{\today}

\begin{document}
\maketitle

\begin{abstract}
We study free infinite divisibility in the two-parameter family of beta
distributions $\{\beta_{p,q}:p,q>0\}$.  Conditional positive definiteness of
free cumulants yields a hierarchy of necessary Hankel conditions.  We
factor the first nontrivial determinant and obtain the explicit necessary
inequality
\[
 2s^3(s+1)+pq\bigl(s^3-7s^2-16s-12\bigr)\geq0,
 \qquad s=p+q.
\]
Its strict reverse defines an open two-dimensional non-freely-infinitely-divisible region not contained in the previously known exclusions.  As a
boundary consequence, we complete the classification of one boundary family:
$\beta_{1/2,q}$ is freely infinitely divisible if and only if
$q\geq3/2$.  The $3\times3$ determinant is also obtained explicitly in the
symmetric variables $s=p+q$ and $u=pq/s^2$.  Finally, exact-rational
$LDL^{\mathsf T}$ certificates show that each leading Hankel test from
$3\times3$ through $12\times12$ strictly enlarges the exclusion supplied by
all preceding leading tests.  In particular, the $4\times4$ test already
detects an open set with $p+q>3$, beyond the range accessible to the
$2\times2$ determinant.  The results are finite-order obstructions rather
than a complete classification; two limiting arguments explain why no fixed member $H_N$ of the leading Hankel hierarchy can provide a uniform obstruction up to the small-parameter boundary.
\end{abstract}

\noindent\textbf{Keywords.}
Free infinite divisibility; beta distribution; free cumulants; Hankel
matrix; conditional positive definiteness.

\noindent\textbf{2020 Mathematics Subject Classification.}
46L54; 60E07.

\section{Introduction}

For $p,q>0$, let $\beta_{p,q}$ be the beta distribution on $[0,1]$,
\begin{equation}\label{eq:density}
 \beta_{p,q}(\dd x)=\frac{x^{p-1}(1-x)^{q-1}}{B(p,q)}
 \mathbf 1_{(0,1)}(x)\dd x.
\end{equation}
A probability measure $\mu$ on $\R$ is freely infinitely divisible (FID) if,
for every $n\in\N$, it is an $n$-fold free additive convolution power of
some probability measure.  The analytic foundations are due to Bercovici
and Voiculescu~\cite{BercoviciVoiculescu1993}; we use the equivalent
conditional-positive-definiteness criterion for free cumulants of compactly
supported measures~\cite[Chapter~13]{NicaSpeicher2006}.

The FID classification of the beta family is only partly known.  Hasebe
proved~\cite{Hasebe2014} that $\beta_{p,q}$ is FID in the regions
\begin{equation}\label{eq:known-fid}
 p,q\geq\frac32,\qquad
 0<p\leq\frac12\ \text{ and }\ p+q\geq2,
 \qquad
 0<q\leq\frac12\ \text{ and }\ p+q\geq2,
\end{equation}
and is not FID if $0<p,q\leq1$, or if either parameter belongs to
\begin{equation}\label{eq:I}
 I=\bigcup_{n\geq1}\left(\frac{2n-1}{2n},\frac{2n}{2n+1}\right)
 \cup\bigcup_{n\geq1}
 \left(\frac{2n+2}{2n+1},\frac{2n+1}{2n}\right).
\end{equation}
There are also exact results on special curves.  The diagonal classification
is $\beta_{t,t}\in\FID$ exactly for $t\geq3/2$
\cite{ArizmendiPerezAbreu2010,Hasebe2014}, while the family
$\beta_{1-1/r,1+1/r}$ on $p+q=2$ is FID exactly for the nondegenerate beta laws, precisely for $1<r\leq2$; the case $r=1$ is interpreted as a
degenerate limiting endpoint~\cite{ArizmendiHasebe2013}.

Hankel positivity is a standard necessary condition for free infinite
divisibility and has previously been used for related power distributions
\cite{ArizmendiPerezAbreu2010,Hasebe2016,MorishitaUeda2018}. Hasebe already used higher-order Hankel determinants directly for beta distributions, obtaining several finite-parameter exclusions and formulating the conjecture that $\beta_{p,q}$ is not freely infinitely divisible whenever one parameter lies in $(1/2,3/2)$~\cite{Hasebe2014}. The present work develops this approach systematically over the full two-parameter plane: we obtain explicit low-order factorizations and exact certificates showing successive strict improvement of the leading Hankel tests. Our
contribution is not a new general criterion, but its systematic application
to the full beta parameter plane.  The main results are as follows.
\begin{enumerate}
 \item We factor the first nontrivial Hankel determinant and describe its
 negative set exactly.  The set contains an open two-dimensional region not
 covered by the known exclusions.
 \item Combining this obstruction with~\eqref{eq:known-fid} and the known
 negative results gives the complete classification
 $\beta_{1/2,q}\in\FID$ if and only if $q\geq3/2$.
 \item We give the next determinant explicitly and prove, by
 exact-rational certificates, that every leading Hankel size from
 $3\times3$ through $12\times12$ gives a strict new exclusion.
\end{enumerate}
To the best of our knowledge, the explicit full-family factorization of $D_1$, the resulting classification of the boundary family $p=1/2$, the symmetric-variable formula for $D_2$, and the exact successive separation of the leading Hankel tests through size 12 have not appeared previously.

The paper is organized as follows.  Section~\ref{sec:hankel} recalls the
Hankel criterion.  Sections~\ref{sec:first} and~\ref{sec:consequences}
derive and apply the first obstruction.  Section~\ref{sec:higher} treats
higher orders.  The appendices record the explicit $3\times3$ formula,
two rigorous limiting mechanisms, and details of the exact computation.

\section{Hankel criteria for free infinite divisibility}
\label{sec:hankel}

Let $m_n(\mu)=\int x^n\mu(\dd x)$ and let $\kappa_n(\mu)$ denote the free
cumulants, determined by
\begin{equation}\label{eq:mc}
 m_n(\mu)=\sum_{\pi\in NC(n)}\prod_{V\in\pi}\kappa_{|V|}(\mu).
\end{equation}
For $N\geq0$, define
\begin{equation}\label{eq:HN}
 H_N(\mu)=\bigl[\kappa_{i+j+2}(\mu)\bigr]_{i,j=0}^{N},
 \qquad D_N(\mu)=\det H_N(\mu),\qquad D_{-1}(\mu)=1.
\end{equation}

\begin{lemma}[Hankel necessary condition]\label{lem:hankel}
If a compactly supported probability measure $\mu$ is FID, then
$H_N(\mu)$ is positive semidefinite for every $N\geq0$.  In particular,
\begin{equation}\label{eq:D1necessary}
 D_1(\mu)=\kappa_2(\mu)\kappa_4(\mu)-\kappa_3(\mu)^2\geq0.
\end{equation}
\end{lemma}

\begin{proof}
For a compactly supported FID measure, the free-cumulant sequence is
conditionally positive definite.  Equivalently, the shifted matrix
$[\kappa_{i+j}]_{i,j\geq1}$ is positive semidefinite on every finite index
set.  Taking the first $N+1$ indices gives the assertion.
\end{proof}

If $D_0,\ldots,D_{N-1}$ are positive, the unpivoted
$LDL^{\mathsf T}$ decomposition of $H_N$ exists through its last pivot and
has diagonal entries
\begin{equation}\label{eq:pivot}
 \ell_j=\frac{D_j}{D_{j-1}},\qquad 0\leq j\leq N.
\end{equation}
Thus positive preceding pivots followed by $\ell_N<0$ provide an exact
certificate that $H_{N-1}$ is positive definite but $H_N$ is not positive
semidefinite.

\section{The first Hankel obstruction}
\label{sec:first}

The moments of~\eqref{eq:density} are
\begin{equation}\label{eq:moments}
 m_n(p,q)=\frac{(p)_n}{(p+q)_n},\qquad n\geq0.
\end{equation}
The moment--cumulant formula gives
\begin{align}
 \kappa_1&=\frac{p}{p+q},\label{eq:k1}\\
 \kappa_2&=\frac{pq}{(p+q)^2(p+q+1)},\label{eq:k2}\\
 \kappa_3&=-\frac{2pq(p-q)}{(p+q)^3(p+q+1)(p+q+2)},\label{eq:k3}\\
 \kappa_4&=\frac{pq\,Q_4(p,q)}
 {(p+q)^4(p+q+1)^2(p+q+2)(p+q+3)},\label{eq:k4}
\end{align}
where
\begin{align*}
 Q_4(p,q)={}&p^3q+6p^3+2p^2q^2-7p^2q+6p^2+pq^3-7pq^2\\
 &{}-18pq+6q^3+6q^2.
\end{align*}

\begin{theorem}[Explicit first obstruction]\label{thm:first}
Let $p,q>0$, put $s=p+q$, and set
\begin{equation}\label{eq:Psi}
 \Psi(p,q)=2s^3(s+1)+pq\bigl(s^3-7s^2-16s-12\bigr).
\end{equation}
Then
\begin{equation}\label{eq:D1factor}
 D_1(\beta_{p,q})=
 \frac{p^2q^2\Psi(p,q)}
 {s^6(s+1)^3(s+2)^2(s+3)}.
\end{equation}
Consequently, $\Psi(p,q)<0$ implies that $\beta_{p,q}$ is not FID.
\end{theorem}

\begin{proof}
Substitution of~\eqref{eq:k2}--\eqref{eq:k4} into~\eqref{eq:D1necessary}
and collection over a common denominator gives a symmetric polynomial in
$p$ and $q$.  Rewriting it in the elementary symmetric variables
$s=p+q$ and $r=pq$ gives precisely~\eqref{eq:D1factor}.  Every factor in
the denominator and the factor $p^2q^2$ are positive.  The final assertion
therefore follows from Lemma~\ref{lem:hankel}.
\end{proof}

\begin{proposition}[Geometry of the obstruction]\label{prop:geometry}
For $p,q>0$ and $s=p+q$,
\begin{equation}\label{eq:region}
 \Psi(p,q)<0
 \quad\Longleftrightarrow\quad
 0<s<3\quad\text{and}\quad
 pq>\frac{2s^3(s+1)}{12+16s+7s^2-s^3}.
\end{equation}
Hence the obstruction is symmetric about $p=q$ and is contained in the
triangle $p+q<3$.
\end{proposition}

\begin{proof}
Write $A(s)=s^3-7s^2-16s-12$.  If $A(s)\geq0$, then $\Psi>0$.
If $A(s)<0$, the bound $pq\leq s^2/4$ gives
\[
 \Psi(p,q)\geq2s^3(s+1)+\frac{s^2}{4}A(s)
 =\frac{s^2}{4}(s-3)(s+2)^2.
\]
Thus negativity is impossible for $s\geq3$.  When $0<s<3$, one has
$12+16s+7s^2-s^3>0$, and rearranging $\Psi<0$ yields~\eqref{eq:region}.
\end{proof}

On the diagonal, $\Psi(t,t)=4t^2(t+1)^2(2t-3)$, so the determinant
recovers the known non-FID range $0<t<3/2$.  On $p+q=2$,
condition~\eqref{eq:region} is $pq>3/4$, equivalently
$1/2<p<3/2$.  These one-dimensional portions agree with the classifications
cited in the introduction and are not new.

\section{Consequences in the parameter plane}
\label{sec:consequences}

\begin{corollary}[A new open non-FID set]\label{cor:open}
The set $\{(p,q):\Psi(p,q)<0\}$ contains a nonempty open subset outside the
previously known non-FID regions and the classified one-dimensional
subfamilies.  In particular,
\[
 \beta_{41/50,\,32/25}\notin\FID.
\]
At this point $\kappa_4>0$, so the conclusion is not a consequence of the
scalar fourth-cumulant obstruction.
\end{corollary}

\begin{proof}
The parameter $41/50$ is in the gap $(4/5,5/6)$ of $I$, and $32/25$ is in
the gap $(5/4,4/3)$.  The point is outside $0<p,q\leq1$, the diagonal, and
$p+q=2$.  Direct exact substitution gives
\[
 \Psi\left(\frac{41}{50},\frac{32}{25}\right)
 =-\frac{8202729}{625000}<0,
 \qquad
 Q_4\left(\frac{41}{50},\frac{32}{25}\right)
 =\frac{3921}{62500}>0.
\]
All relevant inequalities are strict, so they persist on a sufficiently
small open neighborhood.
\end{proof}

\begin{theorem}[The boundary family $p=1/2$]\label{thm:half}
For every $q>0$,
\begin{equation}\label{eq:halfdet}
 D_1(\beta_{1/2,q})=
 \frac{64q^2(2q-3)(20q^3+56q^2+43q-2)}
 {(2q+1)^6(2q+3)^3(2q+5)^2(2q+7)}.
\end{equation}
Consequently,
\begin{equation}\label{eq:halfclass}
 \beta_{1/2,q}\in\FID\quad\Longleftrightarrow\quad q\geq\frac32.
\end{equation}
By symmetry, $\beta_{p,1/2}\in\FID$ if and only if $p\geq3/2$.
\end{theorem}

\begin{proof}
Formula~\eqref{eq:halfdet} follows from~\eqref{eq:D1factor}.  For $q>1$,
the cubic factor is positive, and hence $D_1<0$ for $1<q<3/2$.
Hasebe's result covers $0<q\leq1$ negatively and $q\geq3/2$ positively,
including the endpoint~\cite{Hasebe2014}.  Reflection $x\mapsto1-x$
interchanges the parameters and preserves free infinite divisibility.
\end{proof}

Figure~\ref{fig:regions} compares the first obstruction with the known
regions.  It is only a visualization of exact inequalities and is not used
in any proof.
\begin{figure}[htbp]
 \centering
 \includegraphics[width=0.9\textwidth]{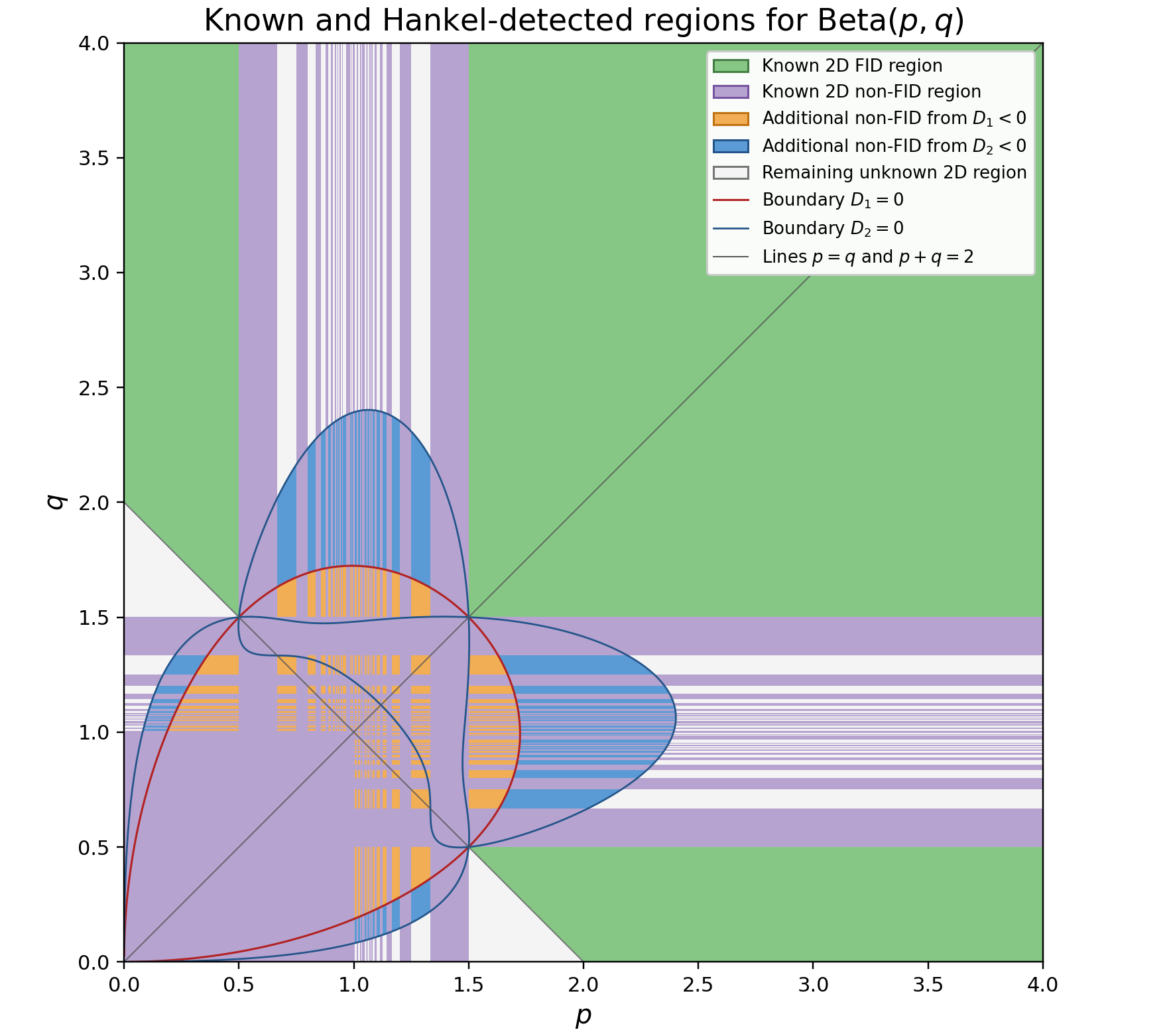}
 \caption{Known FID regions (green), previously known non-FID regions
 (purple), the additional part detected by $D_1<0$ (orange) and $D_2<0$ (blue), and the
 region left undecided by these criteria (gray), on $0<p,q\leq4$.  The
 red curve is $\Psi(p,q)=0$.}
 \label{fig:regions}
\end{figure}

\section{Strict higher-order obstructions}
\label{sec:higher}

Put
\begin{equation}\label{eq:su}
 s=p+q,\qquad u=\frac{pq}{(p+q)^2}\in(0,1/4].
\end{equation}
The next determinant already has the structured form
\begin{equation}\label{eq:D2short}
 D_2(\beta_{p,q})=
 \frac{u^3P_2(s,u)}
 {(s+1)^6(s+2)^4(s+3)^3(s+4)^2(s+5)},
\end{equation}
where $P_2$ is a cubic polynomial in $u$ with coefficients in
$\mathbb Z[s]$; its complete expression is recorded in
Appendix~\ref{app:D2}.  Thus $P_2(s,u)<0$ is an explicit $3\times3$
Hankel obstruction.

The following result shows that the hierarchy does not stabilize at low
order.  Its proof is computer-assisted but exact: every input and every
arithmetic operation belongs to $\mathbb Q$.

\begin{theorem}[Successive strict enlargements]\label{thm:witnesses}
Fix $p=7/10$.  For $2\leq N\leq11$, let $q_N$ be given in
Table~\ref{tab:witnesses}.  At $(p,q_N)$,
\begin{equation}\label{eq:signs}
 D_0,D_1,\ldots,D_{N-1}>0,\qquad D_N<0.
\end{equation}
Consequently, for every $N=2,\ldots,11$, the $(N+1)\times(N+1)$ leading
Hankel test excludes a nonempty open two-dimensional set excluded by none of
the preceding leading Hankel tests.  These open sets can be chosen outside
all non-FID regions in~\eqref{eq:I} and outside $0<p,q\leq1$.
\end{theorem}

\begin{table}[htbp]
\centering
\begin{tabular}{c c c c}
\hline
$N$ & Hankel size & $q_N$ & final pivot $\ell_N$ (decimal display)\\
\hline
2  & $3\times3$   & $2$   & $-2.108479017\times10^{-6}$\\
3  & $4\times4$   & $5/2$ & $-2.009801918\times10^{-7}$\\
4  & $5\times5$   & $3$   & $-3.176453988\times10^{-8}$\\
5  & $6\times6$   & $4$   & $-1.143875589\times10^{-9}$\\
6  & $7\times7$   & $5$   & $-1.379257797\times10^{-10}$\\
7  & $8\times8$   & $7$   & $-6.120402914\times10^{-13}$\\
8  & $9\times9$   & $10$  & $-1.331672286\times10^{-15}$\\
9  & $10\times10$ & $12$  & $-6.974616424\times10^{-17}$\\
10 & $11\times11$ & $15$  & $-2.727627403\times10^{-18}$\\
11 & $12\times12$ & $20$  & $-3.125164564\times10^{-21}$\\
\hline
\end{tabular}
\caption{Exact-rational witnesses at $p=7/10$.  The decimal values show
scale only; their signs are certified by exact rational arithmetic.}
\label{tab:witnesses}
\end{table}

\begin{proof}
For each row, comparing coefficient of $z^n$ in
\[
    M(z)=1+\sum_{n\geq1}\kappa_n z^nM(z)^n,
\]
we compute~\eqref{eq:moments} through order $2N+2$ and use the
formal-series recursion
\begin{equation}\label{eq:recursion}
 \kappa_n=m_n-\sum_{k=1}^{n-1}\kappa_k[z^{n-k}]M(z)^k,
 \qquad M(z)=\sum_{j\geq0}m_jz^j.
\end{equation}
An unpivoted exact-rational $LDL^{\mathsf T}$ decomposition of $H_N$ then
gives $N$ positive pivots followed by the negative pivot displayed in the
table.  By~\eqref{eq:pivot}, this is equivalent to~\eqref{eq:signs}.
Sylvester's criterion shows that $H_{N-1}$ is positive definite while
$H_N$ is not positive semidefinite.

The number $7/10$ lies strictly in the gap $(2/3,3/4)$ of $I$, every
$q_N\geq2$ lies outside $I$, and none of the points belongs to the square
$0<p,q\leq1$.  All determinant and parameter inequalities are strict, so
they persist on an open neighborhood.  Appendix~\ref{app:computation}
describes the reproducible certificate.
\end{proof}

The $4\times4$ witness $(p,q)=(7/10,5/2)$ has $p+q=16/5>3$.
Therefore higher orders do not inherit the restriction in
Proposition~\ref{prop:geometry}.  Another exact witness, in a different gap,
is $(p,q)=(19/16,3)$, where $D_0,D_1,D_2>0$ and $D_3<0$.

\section{Discussion}

Theorems~\ref{thm:first} and~\ref{thm:witnesses} give finite necessary
conditions, not a complete beta classification.  In particular, the witness
table proves neither that the hierarchy is strictly decreasing at every
order nor that every non-FID beta law is eventually separated by a finite
Hankel matrix.

Two rigorous limiting facts clarify why the remaining regimes are difficult.
For fixed $p$, the rescaled beta law $qX_{p,q}$ converges in moments to the
gamma law of shape $p$.  Appendix~\ref{app:limits} shows that the same holds
for every fixed Hankel determinant, giving a conditional large-$q$
exclusion whenever a gamma determinant is negative.  At the opposite
boundary, after reflecting the beta variable and sending one parameter to
zero, every fixed-order rescaled Hankel matrix converges to a strictly
positive beta-moment matrix. Hence no fixed leading Hankel test considered here can settle this boundary regime uniformly. These facts suggest
that a full classification requires either all-order asymptotics or a direct
analytic study of the Voiculescu transform.

\appendix

\section{The explicit \texorpdfstring{$3\times3$}{3 by 3} determinant}
\label{app:D2}

In~\eqref{eq:D2short}, write
\[
 P_2(s,u)=A_0(s)+uA_1(s)+u^2A_2(s)+u^3A_3(s),
\]
where
\begin{align*}
 A_0(s)={}&24s^2(s+1)^4(s+2),\\
 A_1(s)={}&8s(s+1)^2
 (4s^5-17s^4-139s^3-292s^2-318s-216),\\
 A_2(s)={}&4s(s+1)(3s^7-49s^6-125s^5+663s^4+3812s^3\\
 &\hspace{34mm}{}+9004s^2+11472s+5760),\\
 A_3(s)={}&3s^{10}-18s^9+260s^8+1866s^7+1993s^6-16432s^5\\
 &{}-81624s^4-180416s^3-210672s^2-127872s-34560.
\end{align*}
To derive the formula, compute $\kappa_2,\ldots,\kappa_6$ from
\eqref{eq:recursion}, substitute them into $D_2$, and set
$p=as$, $q=(1-a)s$.  The numerator is invariant under $a\mapsto1-a$ and
therefore is a polynomial in $u=a(1-a)$.  Exact collection in powers of
$u$ gives the displayed expression.

\section{Two fixed-order limits}
\label{app:limits}

\begin{lemma}[Scaling and the gamma limit]\label{lem:gamma}
Let $c\mu$ denote the pushforward of $\mu$ under $x\mapsto cx$.  Then
\begin{equation}\label{eq:scaling}
 D_N(c\mu)=c^{(N+1)(N+2)}D_N(\mu).
\end{equation}
For fixed $p>0$ and $N\geq0$,
\begin{equation}\label{eq:gammalimit}
 \lim_{q\to\infty}q^{(N+1)(N+2)}D_N(\beta_{p,q})=D_N(\gamma_p),
\end{equation}
where $\gamma_p$ is the gamma distribution with shape $p$ and unit scale.
If $D_N(\gamma_p)<0$, then $\beta_{p,q}$ is not FID for all sufficiently
large $q$.
\end{lemma}

\begin{proof}
The $n$th free cumulant scales by $c^n$.  Factoring $c^{i+1}$ from row $i$
and $c^{j+1}$ from column $j$ of $H_N$, with indices starting at zero,
gives the exponent $2\sum_{j=0}^{N}(j+1)=(N+1)(N+2)$.
Moreover,
\[
 m_n(qX_{p,q})=q^n\frac{(p)_n}{(p+q)_n}\longrightarrow(p)_n
 =m_n(\gamma_p).
\]
Every cumulant involved in $H_N$ is a polynomial in finitely many moments,
so the matrices and determinants converge.  Strict negativity persists for
all sufficiently large $q$ and invokes Lemma~\ref{lem:hankel}.
\end{proof}

\begin{proposition}[The small-parameter boundary]\label{prop:smallq}
Fix $p>0$ and let $Y_{p,q}=1-X_{p,q}\sim\beta_{q,p}$.  For every fixed
$N\geq0$,
\begin{equation}\label{eq:smallqlimit}
 \frac1qH_N(Y_{p,q})\longrightarrow
 \bigl[B(i+j+2,p)\bigr]_{i,j=0}^{N}\qquad(q\downarrow0).
\end{equation}
The limiting matrix is strictly positive definite.  In particular, for
fixed $p$ and $N$, $H_N(\beta_{p,q})$ is positive definite for all
sufficiently small $q>0$.
\end{proposition}

\begin{proof}
For each $n\geq1$,
\[
 \lim_{q\downarrow0}\frac{m_n(Y_{p,q})}{q}
 =\frac{(n-1)!}{(p)_n}=B(n,p).
\]
Products of two or more nonconstant moments are $O(q^2)$, so the
moment--cumulant formula gives the same first-order limit for $\kappa_n/q$.
This proves~\eqref{eq:smallqlimit}.  Its limiting entries satisfy
\[
 B(i+j+2,p)=\int_0^1x^{i+j+1}(1-x)^{p-1}\dd x,
\]
and hence form the moment matrix of a measure with positive density on
$(0,1)$.  It is strictly positive definite.  Finally, for $n\geq2$,
$\kappa_n(1-X)=(-1)^n\kappa_n(X)$.  Therefore the two shifted Hankel
matrices are related by congruence with the diagonal sign matrix
$\operatorname{diag}(1,-1,1,-1,\ldots)$, so positivity for $Y_{p,q}$ is
equivalent to positivity for $X_{p,q}$.
\end{proof}

For completeness, the determinant of the limiting matrix has the Selberg
product
\begin{equation}\label{eq:selberg}
 \det[B(i+j+2,p)]_{i,j=0}^{N}
 =\frac1{(N+1)!}\prod_{j=0}^{N}
 \frac{\Gamma(j+2)^2\Gamma(p+j)}{\Gamma(p+N+j+2)}>0.
\end{equation}
Formula~\eqref{eq:selberg} is the specialization
$(n,\alpha,\beta,\gamma)=(N+1,2,p,1)$ of the classical Selberg integral \cite{ForresterWarnaar2008} (Eq.~(1.1)). Indeed, Andr\'eief's identity identity \cite{Andreief1886} transforms the Hankel determinant associated with the beta weight $x(1-x)^{p-1}\dd x$ into the corresponding Selberg integral.
Consequently,
$D_N(\beta_{p,q})=q^{N+1}\det[B(i+j+2,p)]+O(q^{N+2})$ after reflection.

\section{Exact computation and reproducibility}
\label{app:computation}

The accompanying script \texttt{higher\_hankel\_analysis.py} implements
\eqref{eq:moments}, the recursion~\eqref{eq:recursion}, symbolic determinant
calculation in $(s,u)$, and exact-rational $LDL^{\mathsf T}$ decomposition.
The witness table is reproduced by
\begin{center}
\texttt{python3 higher\_hankel\_analysis.py --witnesses}.
\end{center}
For every row the script prints all pivot signs and the exact numerator and
denominator of the final pivot.  Thus the decimals in
Table~\ref{tab:witnesses} are not used to decide a sign.  The symbolic
option independently reconstructs the displayed low-order determinant
formulas.  No randomized or floating-point step enters the proofs of
Theorem~\ref{thm:first} or Theorem~\ref{thm:witnesses}.

\section*{Acknowledgments}
Generative-AI systems were used for exploratory symbolic computation and
language editing. All mathematical statements, exact computations, and
references were independently checked by the author.

\bibliographystyle{plain}
\bibliography{beta_hankel_draft}

\end{document}